\documentclass[11pt]{amsart}
\usepackage{amsmath, amssymb, amscd, amsthm, amsfonts, latexsym}
\usepackage{graphicx}
\usepackage{tikz-cd}
\usepackage[colorlinks = true, linkcolor = red, citecolor = blue]{hyperref}
\usepackage{mathtools}
\usepackage{mathrsfs}
\usepackage{enumerate}
\usepackage{mathabx}
\usepackage{setspace}
\usepackage{color}
\usepackage{comment}
\usepackage{tikz}
\usepackage{tikz-cd}
\usepackage{yfonts}

\theoremstyle{plain}
\newtheorem{thm}{Theorem}[section]
\newtheorem{prop}[thm]{Proposition}
\newtheorem{lemm}[thm]{Lemma}
\newtheorem{cor}[thm]{Corollary}

\theoremstyle{definition}

\newtheorem{con}[thm]{Conjecture}

\def\Z{{\mathbb Z}}

\def\cd{\protect\operatorname{cd}}
\def\cat{\protect\operatorname{cat}}

\def\nil{\protect\operatorname{nil}}

\def\Coind{\protect\operatorname{Coind}}

\title{On the LS-category of homomorphism of almost nilpotent groups}
\author{Nursultan Kuanyshov}

\address{Nursultan Kuanyshov, Department of Mathematics, University of Florida, 358 Little Hall, Gainesville, FL 32611-8105, USA}
\email{nkuanyshov@ufl.edu}



\begin{document}
\maketitle
\begin{abstract}
We prove the equality $\cat(\phi)=\cd(\phi)$ for epimorphisms $\phi:\Gamma\to \Lambda$ between torsion-free, finitely generated almost nilpotent groups $\Gamma$ and $\Lambda$. 

In addition, we prove the equality $\cat(\phi)=\cd(\phi)$ for homomorphisms $\phi:\Gamma\to \Lambda$ between torsion-free, finitely generated virtually nilpotent groups. 
\end{abstract}

\section{Introduction}

 The {\em LS-category, denoted as $\cat(X)$, of a topological space $X$} is defined as the minimum number $k$ such that $X$ admits an open cover ${U_0, U_1, \dots, U_k}$ by $k+1$ contractible sets in X. This concept provides a lower bound on the number of critical points for smooth real-valued functions on closed manifolds \cite{LS}, \cite{CLOT}.

Since LS-category is a homotopy invariant, it can be extended to discrete groups $\Gamma$ as $\cat(\Gamma) = \cat(B\Gamma)$, where $B\Gamma$ is a classifying space. Computing the LS-category of spaces, even when considering nice spaces such as manifolds, presents significant challenges \cite{DS}. In the case of groups in the 1950s, Eilenberg and Ganea \cite{EG} established the equality between LS-category and cohomological dimension, $\cat(\Gamma) = \cd(\Gamma)$, for discrete groups $\Gamma$.

We recall the {\em cohomological dimension $\cd(\Gamma)$ of a group $\Gamma$} is defined as the supremum of $k$ such that $H^k(\Gamma, M) \neq 0$, where $M$ is a $\Z\Gamma$-module \cite{Br}. Dranishnikov and Rudyak \cite{DR} showed that $\cd(\Gamma) = \max\{k \mid (\beta_\Gamma)^k \neq 0\}$, where $\beta_\Gamma \in H^1(\Gamma, I(\Gamma))$ is the Berstein-Schwarz class of $\Gamma$ \cite{Be}.

For a map $f: X \to Y$, {\em the LS-category of the map, denoted as $\cat(f)$,} is the minimum number $k$ such that $X$ can be covered by $k+1$ open sets $U_0, U_1, \dots, U_k$ with nullhomotopic restrictions $f|_{U_i}: U_i \to Y$ for all $i$. We define the LS-category $\cat(\phi)$ for a group homomorphism $\phi: \Gamma \to \Lambda$ as $\cat(f)$, where the map $f: B\Gamma \to B\Lambda$ induces the homomorphism $\phi$ on fundamental groups.

Mark Grant \cite{Gr} introduced the cohomological dimension $\cd(\phi)$ of a group homomorphism $\phi: \Gamma \to \Lambda$ as the maximum of $k$ such that there exists a $\Z\Lambda$-module $M$ with a non-zero induced homomorphism $\phi^: H^k(\Lambda, M) \to H^k(\Gamma, M)$. In light of the universality of the Berstein-Schwarz class \cite{DR}, we have $\cd(\phi) = \max\{k \mid \phi(\beta_\Lambda)^k \neq 0\}$. Combining this with the cup-length lower bound for LS-category, we obtain the inequality $\cd(\phi) \leq \cat(\phi)$ for all group homomorphisms.

Considering the Eilenberg-Ganea equality $\cd(\Gamma) = \cat(\Gamma)$, a natural conjecture arises:
\begin{con}\label{con}
For any group homomorphism $\phi:\Gamma\to\Lambda$, it is always the case that $\cat(\phi) = \cd(\phi)$.
\end{con}
Jamie Scott \cite{Sc} investigated this conjecture for geometrically finite groups and proved it for monomorphisms of any groups, as well as for homomorphisms of free and free abelian groups. However, Tom Goodwillie \cite{Gr} provided a counterexample presenting an epimorphism $\phi: G\to\Z^2$ from an infinitely generated group $G$ with $\cd(\phi) = 1$, thus disproving the conjecture.

In our joint work with Dranishnikov \cite{DK}, we reduced the conjecture from arbitrary homomorphisms to epimorphisms and presented a counterexample with an epimorphism between geometrically finite groups. However, we also managed to prove the conjecture for epimorphisms between finitely generated, torsion-free nilpotent groups. It is worth exploring whether the torsion-free restriction can be removed from our result. While dealing with torsion, we investigated a related question: whether it is possible to have a homomorphism of torsion groups $\phi:\Gamma\to\Lambda$ with $\cd(\phi) < \infty$. In our recent paper \cite{Ku}, we provide a negative answer to this question.

In this paper, we prove the conjecture for finitely generated, torsion-free almost nilpotent groups. 
We recall an almost nilpotent group is the extension group of the infinite cyclic group by the torsion-free nilpotent groups. More explicitly, one can think that it is the semidirect product of the torsion-free nilpotent group with infinite cyclic group $\Z\ltimes_{\phi}\Gamma,$ where $\phi:\Z\to Aut(\Gamma)$ is a one-parameter group.

\begin{thm}(Theorem \ref{almost nilpotent})
Let $f:\Z\ltimes_{\phi}\Gamma\to\Z\ltimes_{\nu}\Lambda$ be epimorphism between torsion-free almost nilpotent groups such that $f|_{\Z}=Id.$ Then $\cat(f)=\cd(f)=\cd(\Lambda)+1.$
\end{thm}

Furthermore, we confirm that the conjecture holds for finitely generated, torsion-free virtually nilpotent groups. 
\begin{thm}(Theorem \ref{infranil})
     Let $f:M^{m}\to N^{n}$ be a map of infranilmanifold $M$ that induces epimorphisms $\phi:\pi_{1}(M)\to\pi_{1}(N).$ Then $\cat(\phi)=\cd(\phi).$ In particular, $\cat(\phi)=\cd(\phi)=n.$
\end{thm}

\section{Preliminaries}

We will begin by recalling some classical theorems in group theory and Lie group theory. Discrete groups will be denoted by Greek letters, while Lie groups will be denoted by Latin letters.

\begin{lemm}\label{normal}
Let $\Gamma$ be a group and let $\Gamma'$ be a subgroup of finite index. Then there exists a normal subgroup $\Gamma''$ of $\Gamma$ such that $\Gamma''$ is of finite index in $\Gamma$ and $\Gamma''\leq \Gamma'$.
\end{lemm}

\begin{proof}
Let X be the set of left cosets of $\Gamma'.$ Consider $\phi:\Gamma\to Sym(X)$ given by $\phi(g)(a\Gamma')=(ga)\Gamma'.$ Then $\phi$ is a homomorphism. Take now $\Gamma''=\ker\phi.$ Then $\Gamma''$ is a normal subgroup of $\Gamma$ contained in $\Gamma'$. Finally, $\Gamma/\Gamma''$ is isomorphic to a subgroup of $Sym(X),$ which has order $n!,$ where $n=[\Gamma:\Gamma'].$ Thus, $[\Gamma:\Gamma'']$ is finite.
\end{proof}

A group $\Gamma$ is called a \textit{virtually nilpotent} if it has a nilpotent subgroup $\Gamma'$ of finite index. By Lemma~\ref{normal}, we can always assume that the finite index subgroup is indeed a finite index normal subgroup.

\subsection{Nilpotent groups}
The upper central series of a group $\Gamma$ is a chain of subgroups 
$${e}=Z_{0} \leqslant Z_{1} \leqslant....\leqslant Z_{n} \leqslant.... $$
where $Z_{1}=Z(\Gamma)$ is the center of the group, and $Z_{i+1}$ is the preimage under the canonical epimorphism $\Gamma\to\Gamma/Z_{i}$ of the center of $\Gamma/Z_{i}$.
A group $\Gamma$ is {\em nilpotent} if $Z_{n}=\Gamma$ for some $n$.
The least such $n$ is called  {\em the nilpotency class} of $\Gamma$, denoted $\nil(\Gamma)$. Note that the groups with the nilpotency class one are exactly abelian groups.

The lower central series of a group $\Gamma$ is a chain of subgroups 
$$\Gamma=\gamma_0(\Gamma)\ge \gamma_1(\Gamma)\ge\gamma_2(\Gamma)\ge...$$
defined as $\gamma_{i+1}(\Gamma)=[\gamma_i(\Gamma),\Gamma]$.
It's known that for nilpotent groups $\Gamma$ the nilpotency class $\nil(\Gamma)$ equals the least $n$ for which $\gamma_n(\Gamma)=1$.

\begin{prop}\label{nil}
 Let $\phi:\Gamma\rightarrow \Gamma'$ be an epimorphism. Then $\phi(Z(\Gamma))\subset Z(\Gamma')$ and $\phi(\gamma_i(\Gamma))=\gamma_i(\Gamma')$ for  all $i$.
\end{prop}
\begin{proof}
Straightforward (see for example~\cite{B}, Theorem 5.1.3).
 \end{proof}

\subsection{Construction of infranilmanifold, nilmanifold, and solvmanifold.}
{\em A infranilmanifold} is a closed manifold diffeomorphic to the orbit space $G/\Gamma$ of a simply-connected nilpotent Lie group G of the action of a discrete torsion-free subgroup $\Gamma$ of the semidirect product $G\rtimes K$ where K is a maximal subgroup of $Aut(G).$ If $\Gamma$ lies in the $G$ factor, then the infranilmanifold is called {\em a nilmanifold}. 
Every infranilmanifold $G/\Gamma$ is finitely covered by the nilmanifold $G/\Gamma\cap G.$

Let $\textfrak{g}$ be the Lie algebra of a simply-connected nilpotent Lie group G. It is well-known that the exponential map $exp:\textfrak{g}\to G$ is a global diffeomorphism and the quotient map $G\to G/\Gamma$ is the universal covering map. Hence, every infranilmanifold and nilmanifold are the Eilenberg-MacLane spaces $K(\Gamma, 1).$ By the Mal'cev Theorem~\cite{Ma2} every torsion-free, finitely generated nilpotent group $\Gamma$ can be realized as the fundamental group of some nilmanifold.
The corresponding simply-connected nilpotent Lie group $G$ is obtained as the Mal'cev completion of $\Gamma$. 
Moreover, $\Gamma$ is a lattice in G. In this paper, a lattice in $G$ is a cocompact discrete subgroup. 

A {\em solvmanifold} is a closed manifold diffeomorphic to the quotient space $G/\Gamma$ of a simply-connected solvable Lie group G by discrete cocompact subgroup $\Gamma$ of G.  
It is known that every solvmanifold $G/\Gamma$ can be naturally fibered over a torus with a nilmanifold as fiber:
$N/\Gamma_{N}=(N\Gamma)/\Gamma\to G/\Gamma \to G/(N\Gamma)=T^{k}$
where $N$ is the nilradical of G and $\Gamma_{N}:=\Gamma\cap N$ is a lattice \cite{Ra}. This is known as the Mostow fiber bundle. 
From the Mostow fiber bundle, we obtain that the fundamental group of a solvmanifold $G/\Gamma$ fits into the short exact sequence: $$1\to\Gamma_{N}\to\Gamma\to\Z^{k}\to 1.$$ 
\subsection{Solvable group of a special type: R type}

A Lie Group $G$  is {\em completely solvable} if the corresponding Lie algebra $\textfrak{g}$ satisfies the property that any adjoint linear operator $\text{ad} S:\textfrak{g}\to\textfrak{g},$ $S\in \textfrak{g}$ has only real eigenvalues. In particular, any nilpotent Lie group is completely solvable \cite{Ma1}. 

\begin{thm}(\cite{Sa})\label{Saito}
Let $G$ and $H$ be simply-connected completely solvable Lie groups. Suppose G contains a lattice $\Gamma.$ For any homomorphism $\phi:\Gamma\to H,$ there exists a unique extension to a homomorphism $\bar{\phi}: G\to H.$ 
\end{thm}

\subsection{Berstein-Schwarz cohomology class}
The Berstein-Schwarz class of a discrete group $\Gamma$ is the first obstruction $\beta_{\Gamma}$ to a lift of $B\Gamma=K(\Gamma,1)$ to the universal covering $E\Gamma$. 
Note that $\beta_{\Gamma}\in H^1(\Gamma,I(\Gamma))$ where $I(\Gamma)$ is the augmentation ideal of the group ring $\Z\Gamma$~\cite{Be},\cite{Sch}.

\begin{thm}[Universality~\cite{DR},\cite{Sch}]\label{universal}
For any cohomology class $\alpha\in H^k(\Gamma,L)$, there is a homomorphism of $\Gamma$-modules $I(\Gamma)^k\to L$ such that the induced homomorphism for cohomology takes $(\beta_{\Gamma})^k\in H^k(\Gamma,I(\Gamma)^k)$ to $\alpha$,  where $I(\Gamma)^k=I(\Gamma)\otimes\dots\otimes I(\Gamma)$ and $(\beta_{\Gamma})^k=\beta_{\Gamma}\smile\dots\smile\beta_{\Gamma}$.
\end{thm}

\subsection{Main Lemma} The main results of this paper rely on the following:
\begin{lemm}\label{compact support}\cite{DK} 
For every locally trivial bundle of closed aspherical manifolds  $f:M^{m}\rightarrow N^{n}$ with compact connected fiber $F$ 
the induced homomorphism  $$f^{*}:H^{n}(N;\Z\Gamma) \rightarrow H^{n}(M;\Z\Gamma)$$ is nonzero, where $\Gamma=\pi_1(N)$.
\end{lemm}

In the paper, we use the notation $H^*(\Gamma, A)$ for the cohomology of a group $\Gamma$ with coefficient in $\Gamma$-module $A$. The cohomology groups of a space $X$ with the fundamental group $\Gamma$ we denote as $H^*(X;A)$. Thus, $H^*(\Gamma,A)=H^*(B\Gamma;A)$ where $B\Gamma=K(\Gamma,1)$.

\section{Homomorphisms of almost nilpotent group}

In this section, we prove the conjecture to special types of solvable groups, namely to almost nilpotent groups. We recall that {\em an almost nilpotent Lie group} is a non-nilpotent Lie group with a codimension 1 nilpotent normal subgroup. N is nilradical of a simply-connected Lie group G if the Lie algebra of N, $\textfrak{n}$, is the largest nilpotent ideal contained in the Lie algebra of G, $\textfrak{g}$.   

\begin{lemm}
Let G be the simply connected Lie group, and let N be a nilradical of G. If $dim G/N=1,$ then the fundamental group $\Gamma$ of solvmanifold $G/\Gamma$ is given by semidirect product $\Gamma=\Z\ltimes_{\phi}\Z_{N}.$ Further, the simply connected Lie group G has the form $G=R\ltimes_{\bar{\phi}}N.$ 
\end{lemm}

\begin{proof}
Since $\Gamma$ is a lattice in G, we define $\Gamma_{N}:=\Gamma\cap N.$ It is known $\Gamma_{N}$ is a lattice in N \cite{Ra}. We have the following commutative diagram with exact horizontal rows
$$
\begin{tikzcd}
    1 \arrow[r] & N \arrow[r] & G \arrow[r, "\bar{\pi}"] & G/N \arrow[r]& 1 \\
    1 \arrow[r] & \Gamma_{N} \arrow[r] \arrow[u]& \Gamma \arrow[r, "\pi"] \arrow[u]& \Gamma/\Gamma_{N} \arrow[r] \arrow[u]& 1 \\
\end{tikzcd}
$$
such that $G/N=R^{s}$ and $\Gamma/\Gamma_{N}=\Z^{s}.$ Since $dim G/N=1,$ we get $s=1.$ Since $\Gamma/\Gamma_{N}=\Z$ is free, then the lower in the diagram admits a section $s:\Z\to \Gamma.$ Hence $\Gamma\cong \Z\ltimes_{\phi}\Gamma_{N}.$ 

Since $\Gamma\subset G,$ we can uniquely extend the section $s:\Z\to\Gamma$ to a map $\bar{s}: R=G/N\to G$ by Theorem \ref{Saito}. 

We claim that $\bar{s}$ is a section to $\bar{\pi}:G\to G/N=R.$

This is because $\bar{s}\circ\bar{\pi}$ is extension map of the homomorphism $s\circ\pi:\Gamma/\Gamma_{N} \to \Gamma/\Gamma_{N}$ by the Saito theorem (Theorem \ref{Saito}) again. Since $s\circ\pi=Id|_{\Gamma/\Gamma_{N}}$ and $\bar{Id}$ is the extension of the identity homomorphism $Id:\Gamma/\Gamma_{N}\to\Gamma/\Gamma_{N},$ the uniqueness of the Saito theorem  gives us $\bar{s}\circ\bar{\pi}=\bar{Id}.$

It follows that $G=R\ltimes_{\bar{\phi}} N$ and $\bar{\pi}|_{\Gamma}=\pi.$
\end{proof}

\begin{lemm}\label{Extension map} 
Let $\Gamma=\Z\ltimes_{\phi}\Gamma_{N}$ and $\Lambda=\Z\ltimes_{\nu}\Lambda_{N'}$ be finitely generated, torsion-free almost nilpotent groups. Any epimorphism $f:\Z\ltimes_{\phi}\Gamma_{N}\to\Z\ltimes_{\nu}\Lambda_{N'}$ such that $f|_{\Z}=Id:\Z\to\Z$ can be extended uniquely to a homomorphism $\bar{f}:G\to H,$ where $G:=R\ltimes_{\bar{\phi}}N$ and $H=R\ltimes_{\bar{\nu}}N'.$
\end{lemm}
\begin{proof}
    We have the following commutative diagram where the rows are short exact sequence 
$$
\begin{tikzcd}
1 \arrow[r] & N \arrow[r] & G \arrow[r, "\bar{\pi_{\Gamma}}"] & G/N=R \arrow[r]& 1 \\
1 \arrow[r] & \Gamma_{N} \arrow[r] \arrow[u] \arrow[d, "f|"]& \Gamma \arrow[r, "\pi_{\Gamma}"] \arrow[u] \arrow[d, "f"]& \Gamma/\Gamma_{N}=\Z \arrow[r] \arrow[u] \arrow[d, "Id"]& 1 \\
1 \arrow[r] & \Lambda_{N'} \arrow[r] \arrow[d]& \Lambda \arrow[r, "\pi_{\Lambda}"] \arrow[d]& \Lambda/\Lambda_{N}=\Z \arrow[r] \arrow[d]& 1 \\
1 \arrow[r] & N' \arrow[r] & H \arrow[r, "\bar{\pi_{\Lambda}}"] & H/N'=R \arrow[r] & 1 \\
\end{tikzcd}
$$
Since $\Lambda\subset H,$ the epimorphism $f:\Gamma\to\Lambda$ extends the uniquely $\bar{f}:G\to H$ by the Saito theorem.  Similarly $Id:\Z\to \Z$ extends to a map $\bar{Id}:R\to R.$ Since $Id=\pi_{\Lambda}\circ f\circ s_{\Gamma},$ we get $\bar{Id}=\bar{\pi_{\Lambda}}\circ\bar{f}\circ \bar{s_{\Gamma}}$ by applying the Saito theorem twice. 

By construction, $\bar{f}$ restricted to discrete group $\Gamma$ brings the epimorphism $f:\Gamma\to\Lambda.$

\end{proof}

\begin{lemm}\label{solbundle}
Let $\Gamma=\Z\ltimes_{\phi}\Gamma_{N}$ and $\Lambda=\Z\ltimes_{\nu}\Lambda_{N'}$ be finitely generated, torsion-free almost nilpotent groups. Then every epimorphism $f:\Z\ltimes_{\phi}\Gamma_{N}\to\Z\ltimes_{\nu}\Lambda_{N'}$ such that $f|_{\Z}=Id:\Z\to\Z$ can be realized as a locally trivial bundle of solvmanifolds with the fiber a nilmanifold.
\end{lemm}
\begin{proof}
By Lemma \ref{Extension map}, the kernel of epimorphism $f:\Z\ltimes_{\phi}\Gamma_{N}\to\Z\ltimes_{\nu}\Lambda_{N'}$ such that $f|_{\Z}=Id:\Z\to\Z$ equals the kernel of epimorphism $f|:\Gamma_{N}\to\Lambda_{N'}.$ Similarly, $ker(\bar{f})=ker(\bar{f}|_{N}).$ Let us denote $K:=ker(\bar{f}|_{N})$ and $\pi:=ker(f|_{\Gamma_{N}}).$ Note that $\pi$ is a torsion-free, finitely generated nilpotent group since it is a subgroup of finitely generated, torsion-free nilpotent group $\Gamma_{N}.$ By Mal'cev Theorem ~\cite{Ma2}, $\pi$ can be realized the fundamental group of a nilmanifold $K/\pi,$ where $K$ is a simply connected, nilpotent Lie group, i.e. the Mal'cev completion of the torsion-free nilpotent group $\pi$. 
  Since $G, H$, and $K$ are simply connected completely solvable Lie groups,
  we can apply the Saito theorem to obtain the following commutative diagram   
$$
\begin{tikzcd}
 & K/\pi \arrow[r] & G/\Gamma \arrow[r, "\hat{f}"]& H/\Lambda & \\
1 \arrow[r]  & K \arrow[u, "p_{\pi}"] \arrow[r]& G \arrow[r, "\bar{f}"]  \arrow[u, "p_{\Gamma}"] & H \arrow[r] \arrow[u, "p_{\Lambda}"]& 1 \\
1 \arrow[r] & \pi \arrow[r] \arrow[u]& \Gamma \arrow[r, "f"] \arrow[u]& \Lambda \arrow[r] \arrow[u]& 1. \\
\end{tikzcd}
$$
Here $G=R\ltimes_{\bar{\phi}}N$ and $H=R\ltimes_{\bar{\nu}}N'$ are connected, simply connected almost nilpotent Lie groups whereas $N$ and $N'$ are their nilradicals. 

We claim that the solvmanifold $G/\Gamma$ is a total space of a locally trivial bundle of solvmanifold $H/\Lambda$ with a fiber nilmanifold $K/\pi.$

Since G is the principal $K$-bundle over $H$ \cite{Co} and the projection $p_{\Lambda}: H\to H/\Lambda$ is the universal covering map, we can pick a sufficiently small neighborhood $U$ of a point $x\in H/\Lambda$ such that $p^{-1}_{\Lambda}(U)$ is evenly covered by $\{\bar{U_\lambda}\}_{\lambda \in \Lambda}$. In addition, we require that for each $\lambda$, $\bar{U_\lambda} \times K$ is a local trivialization of the fiber bundle $K \to G \to H$. Let $$X:=\underset{\lambda\in\Lambda}{\coprod}\bar{U_{\lambda}}\times K.$$ Note that the action of the group $\Gamma$ on $X$ translates the summands. The orbit space $X/\Gamma$ is homeomorphic to $$(X/\pi)/\Lambda=\left(\left(\underset{\lambda\in \Lambda}{\coprod}\bar{U}_{\lambda}\times K\right)/\pi\right)/\Lambda=\left(\underset{\lambda\in \Lambda}{\coprod}\bar{U}_{\lambda}\times K/\pi\right)/\Lambda\cong U\times K/\pi.$$ Hence, the preimage $\hat{f}^{-1}(U)$, is homeomorphic to $U\times K/\pi.$

Since  the point $x\in H/\Lambda$ is arbitrary, we get the solvmanifold $G/\Gamma$ is the locally trivial fiber bundle over $H/\Lambda$ with fiber $K/\pi.$

\end{proof}

\begin{cor}\cite{DK}\label{nilbundle}
 Let $\Gamma$ and $\Lambda$ be finitely generated, torsion-free nilpotent groups. Then every epimorphism $\phi:\Gamma\rightarrow \Lambda$ can be realized as a locally trivial bundle of nilmanifolds with the fiber a nilmanifold.
\end{cor}

\begin{thm}{\label{almost nilpotent}}
Let $f:\Z\ltimes_{\phi}\Gamma_{N}\to\Z\ltimes_{\nu}\Lambda_{N'}$ 
be an epimorphism between torsion-free almost nilpotent groups such that $f|_{\Z}=Id.$ Then $$\cat(f)=\cd(f)=\cd(\Lambda_{N'})+1.$$
\end{thm}

\begin{proof}
It is clear by dimensional reason that $$\cd(f)\leq\cat(f)\leq \cd(\Lambda_{N'})+1.$$    

We prove that $\cd(f)=\cd(\Lambda_{N'})+1.$ By Lemma \ref{compact support}, it suffices to show that $f$ can be realized as a fiber bundle over a closed aspherical manifold with compact fiber. Indeed, by Lemma \ref{solbundle} the map $f$ can be realized as a locally trivial fiber bundle over solvmanifold $(R\rtimes_{\bar{\nu}} N')/(\Z\rtimes_{\nu}\Lambda_{N'})$ with fiber nilmanifold $K/\pi$.  This finishes the proof of the theorem. 
\end{proof}

\section{Homomorphism of virtually nilpotent groups}

Let $f:M\to N$ be a map between aspherical manifolds with the fundamental groups $\Gamma$ and $\Lambda$ respectively. We denote by $\phi:\Gamma\to\Lambda$ the induced homomorphism   $\phi:=f_{*}:\pi_{1}(M)\to\pi_{1}(N)$. 

\begin{lemm}{\label{DNE}}
There are no maps from a nilmanifold to a infranilmanifold that induces an epimorphisms of the fundamental groups. 
\end{lemm}

\begin{proof}
Suppose there is a map $f$ from nilmanifold $N$ to infranilmanifold $M$ that induces an epimorphism $\phi:\pi_{1}(N)\to\pi_{1}(M).$ Since $\pi_{1}(N)$ is a nilpotent group, the image $\phi(\pi_{1}(N))$ is a nilpotent group by Proposition ~\ref{nil}. Since $\phi$ is surjective, $\pi_{1}(M)$ must be nilpotent. This is a contradiction since $\pi_{1}(M)$ is a virtually nilpotent group. 
\end{proof}

\begin{thm}\label{infranil}
    Let $f:M^{m}\to N^{n}$ be a map between infranilmanifolds  that induces an epimorphism $\phi:\pi_{1}(M)\to\pi_{1}(N).$ Then $\cat(\phi)=\cd(\phi)=n.$ 
\end{thm}

\begin{proof} Using the well-known inequalities on cat ~\cite{CLOT} we obtain: 
$$\cat(\phi)\leq \min\{\cat(M),\cat{N}\}\leq \dim N=n.$$ 
Since $\cd(\phi)\leq\cat(\phi),$ to obtain the equality $\cd(\phi)=\cat(\phi),$ it suffices to show that $\cd(\phi)=n.$ We consider two steps: 

{\em Step 1.} Let $B\Lambda$ be a nilmanifold.  

Since $\Gamma$ is a virtually nilpotent group, there exists nilpotent subgroup $\Gamma'$ such that the index $|\Gamma:\Gamma'|$ is finite. By Lemma \ref{normal}, we can assume $\Gamma'$ to be normal. We have the following commutative diagrams:
$$
\begin{tikzcd}
\Gamma' \arrow[d, "i_{*}"] \arrow[dr,"h_{*}"]
&  \\
\Gamma \arrow[r,"\phi"] & \Lambda
\end{tikzcd}\ \ \ \ \ \ \ \ \ \ \ \ \ \ \
\begin{tikzcd}
B\Gamma' \arrow[d, "i"] \arrow[dr,"h"]
&  \\
B\Gamma \arrow[r,"f"] & B\Lambda
\end{tikzcd}
$$ where $h:=f|_{B\Gamma'}.$

Since the map $h: B\Gamma'\to\Lambda$ induces an epimorphism $h_{*}:\Gamma'\to\Lambda,$ by Corollary ~\ref{nilbundle} there is a fiber bundle 
of nilmanifolds $B\Gamma'\to B\Lambda$ with a compact fiber. By Lemma~\ref{compact support}, $\cd(h_{*})=n$.
$$
\begin{tikzcd}
H^{n}(B\Gamma';\Z\Lambda) 
&  \\
H^{n}(B\Gamma;\Z\Lambda) \arrow[u,"i^{*}"] & H^{n}(B\Lambda;\Z\Lambda) \arrow[l, "f^{*}"] \arrow[lu, "h^{*}"]
\end{tikzcd}
$$
We claim that that $\cd(\phi)=n$ since $\cd(h_{*})=n.$

Let us pick a element $a\in H^{n}(B\Lambda;\Z\Lambda)$ with $h^{*}(a)\neq 0$. Suppose the contrary that $\cd(\phi)<n,$ then $f^{*}(a)=0.$ This is a contradiction, since $$0=i^{*}(f^{*}(a))=h^{*}(a)\neq 0.$$ Thus, we obtain $\cd(\phi)=n.$

{\em Step 2.} Let $B\Lambda$ be a pure infranilmanifold.
In view of Lemma ~\ref{DNE}, we consider $B\Gamma$ is also a pure infranilmanifold.

Since $\Lambda$ is virtually nilpotent group, there exists a nilpotent subgroup $\Lambda'$ such that the index $|\Lambda:\Lambda'|$ is finite. Therefore, the induce and the co-induced modules coinside, i.e., $$Ind_{\Lambda'}^{\Lambda} M= Coind_{\Lambda'}^{\Lambda} M$$ where $M$ is a $Z\Lambda'$-module. 
By the Shapiro Lemma ~\cite{Br}[proposition 6.2, p. 73], we have the following isomorphisms
$$H^{*}(B\Lambda';\Z\Lambda')\cong H^{*}(B\Lambda; Coind_{\Lambda'}^{\Lambda} \Lambda')=H^{*}(B\Lambda; Ind_{\Lambda'}^{\Lambda}\Z\Lambda')=H^{*}(B\Lambda;\Z\Lambda),$$
since $Ind_{\Lambda'}^{\Lambda}\Z\Lambda'=\Z\Lambda\otimes_{\Z\Lambda'}\Z\Lambda'\cong \Z\Lambda.$

Let $$\alpha:\Coind^{\Lambda}_{\Lambda'} \Z\Lambda'=Hom_{\Lambda'}(\Z\Lambda,\Z\Lambda')\to \Z\Lambda'$$ denote the canonical $\Z\Lambda'$-homomorphism, defined for $g:\Z\Lambda\to \Z\Lambda'$ as $\alpha(g)=g(1)$.

We define the coefficient homomorphism $\beta$ as composition of coefficient homomorphisms $$\beta:\Z\Lambda\cong\Z\Lambda\otimes_{\Z\Lambda'}\Z\Lambda'\cong Hom_{\Z\Lambda'}(\Z\Lambda,\Z\Lambda')\overset{\alpha}{\to}\Z\Lambda'.$$ One can define $\beta$ explicitly: Given $\gamma\in \Lambda,$ $\beta(\gamma)=\gamma$ if $\gamma\in \Lambda'$ and zero otherwise. 

 We get the following commutative diagram 

$$
\begin{tikzcd}
 H^{*}(B\Lambda;\Z\Lambda) \arrow[r,"j^{*}"] \arrow[d, "f^{*}"]  & H^{*}(B\Lambda';\Z\Lambda) \arrow[d,"f'^{*}"] \arrow[r, "\beta_{*}"] & H^{*}(B\Lambda';\Z\Lambda') \arrow[d, "f'^{*}"]  \\
H^{*}(B\Gamma;\Z\Lambda) \arrow[r, "j^{*}"] & H^{*}(B\Gamma';\Z\Lambda) \arrow[r, "\beta_{*}"]  & H^{*}(f^{*}(B\Lambda');\Z\Lambda')  \\
\end{tikzcd}
$$
where $f^{*}(B\Lambda')$ is pull-back of maps $f:B\Gamma\to B\Lambda$ and $i:B\Lambda'\to B\Lambda.$

Since the map $f:B\Gamma\to B\Lambda$ induces an epimorphism of the fundamental groups of manifolds, the pull-back space $f^{*}(B\Lambda')$ is path-connected. Furthermore, $f^{*}(B\Lambda')$ is an infranilmanifold since $B\Lambda'\to B\Lambda$ is a regular covering. Hence, we get that $$f'^{*}:H^{n}(B\Lambda';\Z\Lambda')\to H^{n}(f^{*}(B\Lambda');\Z\Lambda')$$ is nonzero homomorphism by {\em Step 1}, so $\cd(f')=n.$ 

We claim $$f^{*}:H^{n}(B\Lambda;\Z\Lambda)\to H^{n}(B\Gamma;\Z\Lambda)$$ is a nonzero homomorphism. Suppose the contrary and let us pick the element $b\in H^{n}(B\Lambda';\Z\Lambda')$ with $f'^{*}(b)\neq 0.$
Since $$\beta_{*}j^{*}: H^{*}(B\Lambda;\Z\Lambda)\to H^{*}(B\Lambda';\Z\Lambda')$$ is an isomorphism by the Shapiro Lemma, there exists a nonzero element $a\in H^{n}(B\Lambda;\Z\Lambda)$ such that $\beta_{*}j^{*}(a)=b.$ Then $f'^{*}j^{*}(a)=j^{*}f^{*}(a)=j^{*}(0)=0$ since $\cd(f)<n.$  This is a contradiction, as we have the commutative diagram
$$0\neq f'^{*}(b)=f'^{*}\beta_{*}j^{*}(a)=\beta_{*}f'^{*}j^{*}(b)=\beta_{*}j^{*}f^{*}(b)=\beta_{*}j^{*}(0)=0.$$ 
Hence, we prove the claim, which implies $\cd(\phi)=n.$ 
\end{proof}

\section{Acknowledgement} 
The author thanks his advisor Alexander Dranishnikov for extremely helpful discussions and his support.


\footnotesize
\begin{thebibliography}{999999}
\bibliographystyle{alpha}
\bibliography{references} 

\bibitem[B]{B} H. Bechtell, The theory of groups, Addison-Wesley, 1971.

\bibitem[Be]{Be}
I. Berstein,  On the Lusternik-Schnirelmann category of Grassmannians. Math. Proc.
Camb. Philos. Soc. 79  (1976) 129-134.

\bibitem[Br]{Br} K. Brown, Cohomology of Groups. \emph{Graduate Texts in Mathematics},
\textbf{87} Springer, New York Heidelberg Berlin, 1994.

\bibitem[CLOT]{CLOT}
    O. Cornea, G. Lupton, J. Oprea, D. Tanre,
\newblock   { Lusternik-Schnirelmann Category},  AMS,  2003.

\bibitem[Co]{Co} R. Cohen, The topology of fiber bundles lecture notes. Standford University (1998).

\bibitem[DK]{DK} A. Dranishnikov, N. Kuanyshov, On the LS category of homomorphism. Math. Z. to appear, preprint	arXiv:2203.03734 [math.AT]

\bibitem[DS]{DS} A. Dranishnikov, R. Sadykov, The Lusternik–Schnirelmann category of a connected sum, Fundamenta Mathematicae 251 (2020), no. 3, 313-328. 

\bibitem[DR]{DR} A. Dranishnikov, Yu. Rudyak, On the Berstein-Svarc theorem in dimension 2. Math. Proc. Cambridge Philos. Soc. 146 (2009), no. 2, 407-413.

\bibitem[EG]{EG} S. Eilenberg, T. Ganea, {\em On the Lusternik-Schnirelmann Category of Abstract Groups.} Annals of
Mathematics, 65, (1957), 517-518.

\bibitem[Gr]{Gr}
    M. Grant, 
\newblock    { https://mathoverflow.net/questions/89178/cohomological-dimension-of-a-homomorphism}

\bibitem[Ku]{Ku} N. Kuanyshov, On the LS category of homomorphisms of groups with torsion, 
preprint	arXiv:2302.10998 [math.AT]

\bibitem[LS]{LS}  L. Lusternik, L. Schnirelmann, ``Sur le probleme de trois geodesiques fermees sur les surfaces de genre 0", Comptes Rendus de l'Academie des Sciences de Paris, 189: (1929) 269-271.

\bibitem[Ma1]{Ma1} A. I. Mal'tsev,
{\em On a class of homogeneous spaces},
 Izvestiya Rossiiskoi Akademii Nauk. Seriya Matematicheskaya,
  vol. 13 (1949),
  no 3, 201-212.

\bibitem[Ma2]{Ma2} A. I. Mal'tsev,
{\em Nilpotent groups without torsion},
 Izvestiya Rossiiskoi Akademii Nauk. Seriya Matematicheskaya,
  vol. 13 (1949),
  no 1,
  9--32.
  

\bibitem[Ra]{Ra} M. S. Raghunathan,
\newblock{Discrete subgroups of Lie groups},
Springer, 1972.

\bibitem[Sa]{Sa} M. Saito, Sur certains groupes de Lie resolubles, I, II, Sci. Pap. Coll. Gen. Ed. Univ. Tokyo 7 (1957), 1-11, 157-168. 

\bibitem[Sch]{Sch}  A.~Schwarz, The genus of a fibered space. Trudy Moscov. Mat. Obsc. 10, 11 (1961 and 1962), 217-272, 99-126.

\bibitem[Sc]{Sc} J. Scott. {\em On the topological complexity of maps}, Topology and its Applications 314 (2022), Paper No. 108094, 25 pp.


\end{thebibliography}
\end{document}